\documentclass[a4paper,twoside,10pt,leqno,fleqn]{article}
\usepackage{expl3}
\usepackage{amssymb,amsmath,amsthm}
\usepackage[all]{xy}
\usepackage[nottoc, notlof, notlot]{tocbibind}
\usepackage{stmaryrd}
\usepackage{tikz}
\usetikzlibrary{arrows}

\makeatletter
\def\blfootnote{\gdef\@thefnmark{}\@footnotetext}
\makeatother
\title{Complete classification of Brieskorn polynomials up to the arc-analytic equivalence}
\newcommand\shorttitle{Complete arc-analytic classification of Brieskorn polynomials}
\author{Jean-Baptiste Campesato\footnote{\begin{minipage}[t]{\textwidth}Department of Mathematics, Faculty of Science, Saitama University, \\255 Shimo-Okubo, Sakura-ku, Saitama 338-8570, Japan. \\ E-mail address: \url{jbcampesato@mail.saitama-u.ac.jp}\\Research supported by a Japan Society for the Promotion of Science (JSPS) Postdoctoral Fellowship (Short-term) for North American and European Researchers.\end{minipage}}}
\date{August 15, 2017}

\usepackage{ifxetex}
\ifxetex
    \usepackage{polyglossia}
    \setdefaultlanguage[variant=us]{english}
    \usepackage{fontspec}
    \usepackage{unicode-math}
\else
    \usepackage[english]{babel}
    \usepackage{stix}
    \DeclareSymbolFont{calletters}{OMS}{cmsy}{m}{n}
    \DeclareSymbolFontAlphabet{\mathcal}{calletters}
    \usepackage{tgpagella}
    \usepackage[T1]{fontenc}
    \usepackage[utf8x]{inputenc}
\fi

\usepackage{hyperref}
\usepackage{color}
\definecolor{darkred}{rgb}{.5,0,0}
\definecolor{darkgreen}{rgb}{0,.5,0}
\definecolor{darkblue}{rgb}{0,0,.5}
\hypersetup{
    pdfencoding=unicode,
    colorlinks=true,
    linkcolor=darkred,
    citecolor=darkgreen,
    urlcolor=darkblue,
    bookmarks=false,
    pdftoolbar=true,
    pdfmenubar=true,
    pdffitwindow=true,
    pdfstartview={FitH},
    pdfcreator={},
    pdfproducer={},
    pdfnewwindow=true,
}
\makeatletter
    \hypersetup{pdftitle={\shorttitle},pdfauthor={Jean-Baptiste Campesato}}
\makeatother

\usepackage[inner=4cm,outer=3cm,top=2cm,bottom=2cm,includeheadfoot,headsep=.5cm]{geometry}
\usepackage{fancyhdr}
\renewcommand{\thepage}{\arabic{page}}
\fancypagestyle{plain}{
\fancyhf{}

}

\def\footindent{2em}
\makeatletter
\renewcommand\@makefntext[1]{\leftskip=\footindent\hskip-\footindent\@makefnmark#1}
\makeatother
\usepackage{perpage}
\MakePerPage{footnote}
\makeatletter
\ifxetex
\newcommand*{\fnsymbolsingle}[1]{\ensuremath{\ifcase#1\or\star\or\dagger\or\ddagger\or\textsection\or\|\or\textparagraph\or\else\@ctrerr\fi}}
\else
\newcommand*{\fnsymbolsingle}[1]{\ensuremath{\ifcase#1\or\star\or\dagger\or\ddagger\or\mathsection\or\|\or\mathparagraph\or\else\@ctrerr\fi}}
\fi
\makeatother
\usepackage{alphalph}
\newalphalph{\fnsymbolmult}[mult]{\fnsymbolsingle}{}

\makeatletter
\renewcommand{\maketitle}{
  \newpage
  \null
  \thispagestyle{plain}
  \begin{center}
    {\LARGE \@title \par}
    \vskip 1.5em
    {\large
      \lineskip .5em
        \@author
      \par}
    \vskip 1em
    {\large \@date}
  \end{center}
  \par
  \vskip 1.5em}
\makeatother

\theoremstyle{plain}
\newtheorem{thm}{Theorem}[section]
\newtheorem{prop}[thm]{Proposition}
\newtheorem{cor}[thm]{Corollary}
\newtheorem{lemma}[thm]{Lemma}
\theoremstyle{definition}
\newtheorem{defn}[thm]{Definition}

\newtheorem{eg}[thm]{Example}
\newtheorem{rem}[thm]{Remark}

\newtheorem{notation}[thm]{Notation}

\usepackage{enumitem}
\setlist[itemize]{labelindent=.6em, itemindent=1em, leftmargin=!, label=\textbullet}

\usepackage{ifmtarg}
\makeatletter
  \newcommand{\TODO}[1]{\@ifmtarg{#1}{\emph{\textbf{TODO}}~}{\emph{\textbf{TODO:}~#1~}}}
\makeatother

\newcommand{\ac}{\operatorname{ac}}

\renewcommand{\d}{\mathrm{d}}
\newcommand{\X}{\mathfrak X}

\newcommand{\mon}{\mathrm{mon}}
\newcommand{\pr}{\operatorname{pr}}

\newcommand{\AS}{\mathcal{AS}}

\newcommand{\id}{\operatorname{id}}
\newcommand{\M}{\mathcal M}

\newcommand{\bigast}{\mathop{\scalebox{1.5}{\raisebox{-0.2ex}{$*$}}}}
\newcommand{\lcm}{\operatorname{lcm}}

\newcommand{\naive}{\mathrm{naive}}

\def\acts{\ensuremath{\rotatebox[origin=c]{-90}{$\circlearrowright$}}}
\newcommand{\sgn}{\operatorname{sgn}}

\makeatletter
\newcommand{\vast}{\bBigg@{4}}
\newcommand{\Vast}{\bBigg@{5}}
\makeatother

\usepackage{graphicx}
\def\acts{\ensuremath{\rotatebox[origin=c]{-90}{$\circlearrowright$}}}

\begin{document}
\maketitle

\begin{abstract}
It has been recently proved that the arc-analytic type of a singular Brieskorn polynomial determines its exponents. This last result may be seen as a real analogue of a theorem by E. Yoshinaga and M. Suzuki concerning the topological type of complex Brieskorn polynomials. In the real setting it is natural to investigate further by asking how the signs of the coefficients of a Brieskorn polynomial change its arc-analytic type.

The aim of the present paper is to answer this question by giving a complete classification of Brieskorn polynomials up to the arc-analytic equivalence. The proof relies on an invariant of this relation whose construction is similar to the one of Denef--Loeser motivic zeta functions.

The classification obtained generalizes the one of Koike--Parusiński in the two variable case up to the blow-analytic equivalence and the one of Fichou in the three variable case up to the blow-Nash equivalence.
\end{abstract}

\tableofcontents

\section{Introduction}
In order to obtain a classification of real singularities with no continuous moduli, T.-C. Kuo \cite{Kuo85} introduced the blow-analytic equivalence. He proved that it is an equivalence relation for real analytic function germs with no continuous moduli for isolated singularities.

S. Koike and A. Parusiński \cite{KP03} introduced invariants of the blow-analytic equivalence constructed similarly to Denef--Loeser motivic zeta functions \cite{DL98} but realized through the Euler characteristic with compact support. Using these invariants, together with the Fukui invariants \cite{Fuk97}, they classified entirely Brieskorn polynomials\footnote{i.e. polynomials of the form $\pm x^p\pm y^q$.} in two variables up to the blow-analytic equivalence \cite[Theorem 6.1]{KP03}. However, for the three variable case, these invariants do not allow them to distinguish the following Brieskorn polynomials \cite[Theorem 7.3]{KP03}: $x^p+y^{kp}+z^{kp}$ and $-x^p-y^{kp}-z^{kp}$ where $p$ is even.

Then G. Fichou \cite{Fic05} introduced the blow-Nash equivalence as a semialgebraic version of the blow-analytic equivalence allowing him to construct richer invariants. These invariants are still constructed analogously to Denef--Loeser motivic zeta functions, but, this time, realized through the virtual Poincaré polynomial, a real analogue of the Hodge--Deligne polynomial due to C. McCrory and A. Parusiński \cite{MP03,MP11}. Using these invariants, G. Fichou classified entirely Brieskorn polynomials\footnote{i.e. polynomials of the form $\pm x^p\pm y^q\pm z^r$.} in three variables up to the blow-Nash equivalence by showing that the Fichou motivic zeta functions of the above cited Brieskorn polynomials are different \cite[Example 4.10]{Fic05}.

In this paper, we are going to work with a slightly different framework by considering the classification of Brieskorn polynomials up to the arc-analytic equivalence. This relation was introduced in \cite{JBC2}. It is a characterization of the blow-Nash equivalence in simpler terms. It also allows one to prove it is an equivalence relation for Nash function germs, which was expected but not known yet for the blow-Nash equivalence. Moreover, using recent results due to A. Parusiński and L. Păunescu \cite{PP}, one may prove it has no continuous moduli even without the isolated singularitiy condition.

It is already known that the arc-analytic type of a singular Brieskorn polynomial determines its exponents \cite[Corollary 8.4]{JBC2}. This last result is a real analogue of a theorem by E. Yoshinaga and M. Suzuki \cite{YS78} concerning the topological type of complex Brieskorn polynomials. In the complex case, it is possible to assume that the coefficients of a Brieskorn polynomial are all equal to $1$ using a linear change of variables. This is not possible with real numbers where we can only assume that the coefficient of a monomial of even degree is either $1$ or $-1$. So, it is natural to wonder what is the impact of the signs of the coefficients of a Brieskorn polynomial on its arc-analytic type.

The aim of this papier is to answer this question by giving a complete classification of Brieskorn polynomials up to the arc-analytic equivalence.

The main result is divided into two parts. First, it states that the motivic invariant introduced in \cite{JBC2} is a complete arc-analytic invariant for Brieskorn polynomials. Next, it explains how to determine the arc-analytic type of a singular Brieskorn polynomial from its exponents and coefficients.

More precisely, we show that two singular Brieskorn polynomials are arc-analytically equivalent if and only if they share the same exponents and if their coefficients of even degree not multiple of an odd exponent share the same signs. In order to prove this theorem, we are going to show that we may recover the concerned signs from the real motivic zeta function introduced in \cite{JBC2}, which is enough since it is an invariant of the arc-analytic equivalence. The proof relies on formulae given in the last section for the virtual Poincaré polynomials of the fibers of a Brieskorn polynomial.

In addition to completing the arc-analytic classification of Brieskorn polynomials, this result highlights some differences between the arc-analytic classification and other real analytic function germ classifications. On the one hand, the arc-analytic classification is finer than the $C^0$-(right-)equivalence. Indeed, contrary to the complex case, the $C^0$ class of a Brieskorn polynomial doesn't determine its exponents. For example, if a Brieskorn polynomial has an odd exponent, it is $C^0$-equivalent to $\pr_1:(x_1,\ldots,x_d)\mapsto x_1$ and therefore, $x^2-y^3$ is $C^0$-equivalent to $x^2-y^9$. On the other hand, the arc-analytic classification is coarser than the analytic-(right-)equivalence. Let $p$ be an odd number, then, by Proposition \ref{prop:KPgen}, $x^p+y^{kp}$ and $x^p-y^{kp}$ are arc-analytically equivalent, hence the sign of the coefficient in front of $y^{kp}$ is not determined by the arc-analytic type. This phenomenon doesn't hold anymore for the analytic equivalence or the bi-Lipschitz equivalence (see \cite[Example 6.3]{HP04}). Nevertheless, these latter admit continuous moduli contrary to the arc-analytic equivalence.

One could investigate further by obtaining a complete arc-analytic classification of weighted homogeneous polynomials. It is currently known that the arc-analytic type of a weighted homogeneous polynomial which is also Newton non-degenerate and convenient determines its weights \cite{JBC4}, but nothing is known about the coefficients. \\

\noindent\textbf{Acknowledgements.} I am sincerely grateful to Adam Parusiński for our fruitful discussions during the preparation of this article. I express my gratitude and thanks to Toshizumi Fukui who warmly welcomed me in Saitama University where this work has been carried out.

\section{Recollection}
We refer the reader to \cite[\S7]{JBC3} for a survey concerning the material covered in this section.
\subsection{The arc-analytic equivalence}
\begin{defn}[{\cite[Definition 7.5]{JBC2}}]
Two Nash\footnote{A Nash function is a smooth function with semialgebraic graph. Such a function is necessarily real analytic.} function germs $f,g:(\mathbb R^d,0)\rightarrow(\mathbb R,0)$ are said to be arc-analytically equivalent if there exists a semialgebraic homeomorphism $h:(\mathbb R^d,0)\rightarrow(\mathbb R^d,0)$ such that
\begin{enumerate}[label=(\roman*),nosep]
\item $f=g\circ h$
\item $h$ is arc-analytic\footnote{This is a notion due to K. Kurdyka \cite{Kur88}.}, i.e. it maps real analytic arcs to real analytic arcs by composition.
\item There exists $c>0$ such that $|\det\d h|>c$ where $\d h$ is defined\footnote{Indeed, K. Kurdyka proved that a semialgebraic arc-analytic map is real analytic outside a set of codimension at least 2.}.
\end{enumerate}
\end{defn}

\begin{rem} We recall the following properties of the arc-analytic equivalence.
\begin{itemize}[nosep]
\item It is an equivalence relation \cite[Proposition 7.7]{JBC2}: indeed, under the assumptions of the definition, $h^{-1}$ is arc-analytic and there exists $c'>0$ such that $|\det\d h^{-1}|>c'$, see \cite{FKP10} and \cite[Corollary 3.6]{JBC1}.
\item It coincides with the blow-Nash equivalence of Fichou \cite[Proposition 7.9]{JBC2}.
\item It has no continuous moduli by \cite{PP}, see \cite[Theorem 7.5]{JBC3}.
\end{itemize}
\end{rem}

The following result will be useful. It is a weak version of an arc-analytic analog of a theorem by T. Fukui and L. Păunescu for the blow-analytic equivalence \cite[Theorem (0.2)]{FP00}.
\begin{thm}[{\cite[Corollary 4.5]{Fic05}}]\label{thm:TFWeighted}
Let $F:(\mathbb R^d\times I,\{0\}\times I)\rightarrow(\mathbb R,0)$ be Nash where $I$ is a compact interval. For $t\in I$, set $f_t=F(\cdot,t):(\mathbb R^d,0)\rightarrow(\mathbb R,0)$. Assume that for all $t\in I$, $f_t$ admits an isolated singularity at the origin, it is weighted homogenous and that the weight system doesn't depend on $t$. Then for all $t,t'\in I$, $f_t$ and $f_{t'}$ are arc-analytically equivalent.
\end{thm}

\subsection{A motivic invariant of the arc-analytic equivalence}

\begin{defn}[{\cite[\S4.2]{Par04}}]
An $\AS$-set is a semialgebraic subset $A\subset\mathbb P^n_\mathbb R$ such that given a real analytic arc $\gamma:(-1,1)\rightarrow\mathbb P^n_\mathbb R$ satisfying $\gamma(-1,0)\subset A$ there exists $\varepsilon>0$ such that $\gamma(0,\varepsilon)\subset A$.
\end{defn}

\begin{rem}[{\cite[\S4.2]{Par04}}]
The $\AS$-subsets of $\mathbb P^n_\mathbb R$ form the boolean algebra spanned by semialgebraic arc-symmetric\footnote{A subset $S$ of a real analytic manifold $M$ is arc-symmetric if given a real analytic arc on $M$, either this arc is entirely included in $S$ or it meets $S$ at isolated points only. This is a a notion due to K. Kurdyka \cite{Kur88}.} subsets of $\mathbb P^n_\mathbb R$. Particularly, $\AS$ is closed under $\cup,\cap,\setminus$.
\end{rem}

The following remark will be useful for the computations in Section \ref{sub:SomeLemmas}.
\begin{rem}[{\cite[Proposition 4.6 \& Theorem 2.8]{Par04}}]
Real Zariski-constructible sets are $\AS$-sets.
\end{rem}

\begin{defn}
We denote by $K_0(\AS)$ the free abelian group spanned by the symbols $[A]$, $A\in\AS$, modulo:
\begin{enumerate}[label=(\roman*),nosep]
\item If there is a bijection $A\rightarrow B$ whose graph is in $\AS$ then $[A]=[B]$.
\item If $B$ is a closed $\AS$-subset of $A$ then $[A]=[A\setminus B]+[B]$.
\end{enumerate}
Moreover, $K_0(\AS)$ has a ring structure induced by the cartesian product:
\begin{enumerate}[resume,label=(\roman*),nosep]
\item $[A\times B]=[A][B]$.
\end{enumerate}
We denote by $0=[\varnothing]$ the class of the empty set which is the unit of the addition, by $1=[\{*\}]$ the class of the point which is the unit of the product and by $\mathbb L_\AS=[\mathbb R]$ the class of the affine line. We set $\M_\AS=K_0(\AS)\left[\mathbb L_\AS^{-1}\right]$ the localization of $K_0(\AS)$ with respect to $\left\{\mathbb L_\AS^{i},\,i\in\mathbb N\right\}$.
\end{defn}

\begin{thm}[{\cite{MP03}\cite{Fic05}\cite{MP11}}]\label{thm:vpp}
There exists a unique ring morphism $\beta:K_0(\AS)\rightarrow\mathbb Z[u]$, called the \emph{virtual Poincaré polynomial}, such that if $A\in\AS$ is compact and non-singular then $\beta([A])=\sum_i\dim H_i(A,\mathbb Z_2)u^i$.

Moreover, the virtual Poincaré polynomial encodes the dimension since if $A\in\AS$ is nonempty then $\deg\beta([A])=\dim A$ and the leading coefficient is positive.
\end{thm}

\begin{eg}
$\beta([\mathbb R])=\beta([\mathbb P^1_\mathbb R\setminus\{*\}])=\beta([\mathbb P^1_\mathbb R])-\beta([*])=u+1-1=u$
\end{eg}

\begin{rem}
The virtual Poincaré polynomial induces a ring morphism $\beta:\M_\AS\rightarrow\mathbb Z[u,u^{-1}]$.
\end{rem}

\begin{rem}
Using the cell decomposition property of semialgebraic sets, one may prove that every additive invariant of the semialgebraic sets up to semialgebraic homeomorphisms factorizes through the Euler characteristic with compact support, see \cite{Qua01}. It is then not possible to recover the dimension since, for example, $\chi_c(S^1)=0$. Hence, working with $\AS$-sets allows us to use the virtual Poincaré polynomial which is an additive invariant encoding more information. Notice also that $\beta([A])_{|u=-1}=\chi_c(A)$.
\end{rem}

The following Grothendieck group is a real analogue of the one defined by Guibert--Loeser--Merle \cite{GLM} that fits our context.
\begin{defn}[{\cite[Definition 3.4, Notation 3.5]{JBC2}}]
We denote by $K_0(\AS_\mon^n)$ the free abelian group spanned by symbols $[\varphi_X:\mathbb R^*\acts X\rightarrow\mathbb R^*]$, where $X\in\AS$, the graph $\Gamma_{\varphi_X}\in\AS$, the graph of the action $\Gamma_{\mathbb R^*\times X\rightarrow X}\in\AS$ and $\varphi_X(\lambda\cdot x)=\lambda^n\varphi_X(x)$, modulo the relations:
\begin{enumerate}[label=(\roman*),nosep]
\item If there is a $\mathbb R^*$-equivariant bijection $h:X\rightarrow Y$ with $\AS$-graph such that $\varphi_X=\varphi_Y\circ h$ then $$[\varphi_X:\mathbb R^*\acts X\rightarrow\mathbb R^*]=[\varphi_Y:\mathbb R^*\acts Y\rightarrow\mathbb R^*]$$
\item If $Y\subset X$ is a closed $\mathbb R^*$-invariant $\AS$-subset then $$[\varphi_X:\mathbb R^*\acts X\rightarrow\mathbb R^*]=[\varphi_{X|X\setminus Y}:\mathbb R^*\acts X\setminus Y\rightarrow\mathbb R^*]+[\varphi_{X|Y}:\mathbb R^*\acts Y\rightarrow\mathbb R^*]$$
\item Fix $\varphi_Y:\mathbb R^*\acts_\tau Y\rightarrow\mathbb R^*$ satisfying the required conditions to be a symbol. Set $\psi=\varphi_Y\pr_Y:Y\times\mathbb R^m\rightarrow\mathbb R^*$. If $\sigma$ and $\sigma'$ are two actions of $\mathbb R^*$ on $Y\times\mathbb R^m$ which are liftings of $\tau$ then $\psi:\mathbb R^*\acts_{\sigma}(Y\times\mathbb R^m)\rightarrow\mathbb R^*$ and $\psi:\mathbb R^*\acts_{\sigma'}(Y\times\mathbb R^m)\rightarrow\mathbb R^*$ satisfy the conditions to be symbols and we add the relation $$\left[\psi:\mathbb R^*\acts_{\sigma}(Y\times\mathbb R^m)\rightarrow\mathbb R^*\right]=\left[\psi:\mathbb R^*\acts_{\sigma'}(Y\times\mathbb R^m)\rightarrow\mathbb R^*\right]$$
\end{enumerate}
The fiber product over $\mathbb R^*$ induces a structure of ring by adding the following relation:
\begin{enumerate}[label=(\roman*),nosep,resume]
\item $[\mathbb R^*\acts(X\times_{\mathbb R^*}Y)\rightarrow\mathbb R^*]=[\varphi_X:\mathbb R^*\acts X\rightarrow\mathbb R^*][\varphi_Y:\mathbb R^*\acts Y\rightarrow\mathbb R^*]$ where the action of $\mathbb R^*$ on $X\times_{\mathbb R^*}Y$ is the diagonal action.
\end{enumerate}
The cartesian product induces a structure of $K_0(AS)$-algebra by adding the following relation:
\begin{enumerate}[label=(\roman*),nosep,resume]
\item $[A][\varphi_X:\mathbb R^*\acts_\tau X\rightarrow\mathbb R^*]=[\varphi_X\circ\pr_X:\mathbb R^*\acts_\sigma A\times X\rightarrow\mathbb R^*]$ where $\lambda\cdot_\sigma(a,x)=(a,\lambda\cdot_\tau x)$.
\end{enumerate}
\ \\
We set $K_0(\AS_\mon)=\varinjlim K_0(\AS_\mon^n)$ where the direct system is defined as follows. For $n=km$ with $k\in\mathbb N_{>0}$, we set $\theta_{mn}:K_0(\AS_\mon^m)\rightarrow K_0(\AS_\mon^n)$ defined by $[\varphi_X:\mathbb R^*\acts_m X\rightarrow\mathbb R^*]\mapsto[\varphi_X:\mathbb R^*\acts_n X\rightarrow\mathbb R^*]$ where $\lambda\cdot_n x=\lambda^k\cdot_m x$.
\end{defn}

\begin{notation}
We denote by $0=[\emptyset]$ the class of the empty set which is the unit of the addition, by $\mathbb 1=[\id:\mathbb R^*\acts\mathbb R^*\rightarrow\mathbb R^*]$ the class of the identity which is the unit of the product and by $$\mathbb L=\mathbb L_\AS\mathbb 1=[\pr_2:\mathbb R^*\acts\mathbb R\times\mathbb R^*\rightarrow\mathbb R^*]$$ the class of the affine line.

We set $\M=K_0(\AS_\mon)[\mathbb L^{-1}]$. Notice that $\M$ has a natural structure of $\M_\AS$-algebra.
\end{notation}

\begin{prop}[The forgetful morphism {\cite[\S3]{JBC2}}]
There exists a unique morphism of $\M_\AS$-modules $\overline{\vphantom{I}\hspace{1mm}\cdot\hspace{1mm}}:\M\rightarrow\M_\AS$ such that $\overline{\left[\varphi_X:\mathbb R^*\acts X\rightarrow\mathbb R^*\right]}=[X]$.
\end{prop}

\begin{prop}[{\cite[Proposition 4.16]{JBC2}}]
For $\varepsilon\in\{+,-\}$, there exists a unique morphism of $\M_\AS$-algebras $F^\varepsilon:\M\rightarrow\M_\AS$ such that $F^\varepsilon\left(\left[\varphi_X:\mathbb R^*\acts X\rightarrow\mathbb R^*\right]\right)=\left[\varphi_X^{-1}(\varepsilon1)\right]$.
\end{prop}

\begin{rem}
The forgetful morphism is not compatible with the ring structures since the one on $\M$ is induced by the fiber product whereas the one on $\M_\AS$ is induced by cartesian product. Particularly $\beta\left(\overline{\mathbb 1}\right)=u+1\neq1=\beta(1)$.

However, the morphisms $F^\varepsilon$ are compatible with the ring structures since the fiber product over one point coincides with the cartesian product.
\end{rem}

\begin{defn}[{\cite[Definition 4.2]{JBC2}}]\label{defn:ZetaFun}
We define the motivic zeta function of a Nash function germ $f:(\mathbb R^d,0)\rightarrow(\mathbb R,0)$ by $$Z_f(T)=\sum_{n\ge1}\left[\ac_f^n:\mathbb R^*\acts\X_n(f)\rightarrow\mathbb R^*\right]\mathbb L^{-nd}T^n\in\M\llbracket T\rrbracket$$ where
\begin{itemize}[nosep]
\item $\X_n(f)=\left\{\gamma(t)=a_1t+\ldots+a_nt^n,\,a_i\in\mathbb R^d,\,f(\gamma(t))=ct^n+\cdots,\,c\neq0\right\}$
\item $\ac_f^n:\X_n(f)\rightarrow\mathbb R^*$ is the angular component map defined by $\ac_f^n(\gamma)=\ac(f\circ\gamma):=c$
\item $\lambda\cdot\gamma(t)=\gamma(\lambda t)$
\end{itemize}
so that $\left[\ac_f^n:\mathbb R^*\acts\X_n(f)\rightarrow\mathbb R^*\right]\in K_0(\AS_\mon^n)$.
\end{defn}

\begin{thm}[{\cite[Theorem 7.11]{JBC2}}]\label{thm:Zaainv}
If $f,g:(\mathbb R^d,0)\rightarrow(\mathbb R,0)$ are two arc-analytically equivalent Nash function germs then $Z_f(T)=Z_g(T)$.
\end{thm}

\begin{defn}[{\cite[Definition 6.11]{JBC2}}]
We define the modified zeta function of a Nash function germ $f:(\mathbb R^d,0)\rightarrow(\mathbb R,0)$ by $$\tilde Z_f(T)=Z_f(T)-\frac{\mathbb 1-Z_f^\naive(T)}{\mathbb 1-T}+\mathbb 1\in\M\llbracket T\rrbracket$$ where $Z_f^\naive(T)$ is obtained from $Z_f(T)$ by applying coefficientwise $\M\rightarrow\M,\,\alpha\mapsto\overline\alpha \mathbb 1$.
\end{defn}

\begin{rem}
By its very definition, the modified zeta function is also an arc-analytic invariant. Actually $Z_f(T)$ and $\tilde Z_f(T)$ encode the same information \cite[Corollary 6.14]{JBC2}:
$$Z_f(T)=\tilde Z_f(T)+\frac{\mathbb 1-\mathbb L^{-1}\tilde Z_f^\naive(T)}{\mathbb 1-\mathbb L^{-1}T}-\mathbb 1$$
\end{rem}

\begin{prop}[{\cite[Notation 6.1]{JBC2}}]
There exists a unique $K_0(\AS)$-bilinear map, called the convolution product, $*:K_0(\AS_\mon^m)\times K_0(\AS_\mon^n)\rightarrow K_0(\AS_\mon^{mn})$ satisfying the following relation on symbols
\begin{align*}
&\left[\varphi_X:\mathbb R^*\acts X\rightarrow\mathbb R^*\right]*\left[\varphi_Y:\mathbb R^*\acts Y\rightarrow\mathbb R^*\right]\\
&\quad\quad=-\left[\varphi_X+\varphi_Y:\mathbb R^*\acts_\neq\left(X\times Y\setminus(\varphi_X+\varphi_Y)^{-1}(0)\right)\rightarrow\mathbb R^*\right]\\
&\quad\quad\quad+\left[\pr_2:\mathbb R^*\acts_=\left((\varphi_X+\varphi_Y)^{-1}(0)\times\mathbb R^*\right)\rightarrow\mathbb R^*\right]
\end{align*}
where $\lambda\cdot_\neq(x,y)=(\lambda^n\cdot x,\lambda^m\cdot y)$ and $\lambda\cdot_=(x,y,r)=(\lambda^n\cdot x,\lambda^m\cdot y,\lambda^{mn}r)$.
\end{prop}

\begin{rem}
This induces a convolution product $*:\M\times\M\rightarrow\M$ which is $\M_\AS$-bilinear, commutative, associative and whose unit is $\mathbb 1$.
\end{rem}

\begin{thm}[The convolution formula {\cite[Theorem 6.15]{JBC2}}]\label{thm:convolutionf}
Let $f_i:(\mathbb R^{d_i},0)\rightarrow(\mathbb R,0)$ be a Nash function germ for $i=1,2$. Define $f_1\oplus f_2:(\mathbb R^{d_1+d_2},0)\rightarrow(\mathbb R,0)$ by $f_1\oplus f_2(x_1,x_2)=f_1(x_1)+f_2(x_2)$. Then $$\tilde Z_{f_1\oplus f_2}(T)=-\tilde Z_{f_1}(T)\circledast\tilde Z_{f_2}(T)$$ where $\circledast$ consists in applying $*$ coefficientwise.
\end{thm}

\section{Statement of the main theorem}
\begin{defn}
A polynomial $f\in\mathbb R[x_1,\ldots,x_d]$ is said to be a Brieskorn polynomial if it is of the following form $$f(x)=\sum_{i=1}^d\varepsilon_ix_i^{k_i},\,\varepsilon_i\neq0,\,k_i\ge1$$
\end{defn}

\begin{notation}\label{not:BrieskSimpl}
Since we are only interested in the arc-analytic classification of Brieskorn polynomials, we can perform the following easy simplifications:
\begin{itemize}[nosep]
\item Without modifying the arc-analytic type, we may replace $\varepsilon_i$ by its sign $\sgn(\varepsilon_i)$ using a linear change of variables\footnote{When $k_i$ is odd, we could also have fixed that $\varepsilon_i=1$.}.
\item Without modifying the arc-analytic type, we may reorder the variables and assume that the exponents are ordered: $$k_1\le k_2\le\cdots\le k_d$$
\item Still by reordering the variables, we may assume that if some exponents are equal, we put the positive coefficients first:
$$k_i=k_{i+1}=\cdots=k_{i+m}\Rightarrow\varepsilon_i\ge\varepsilon_{i+1}\ge\cdots\ge\varepsilon_{i+m}$$
\item A Brieskorn polynomial $f$ is non-singular if and only if there exists $i$ such that $k_i=1$ if and only if $\tilde Z_f(T)=0$ if and only if $Z_f(T)=\sum_{n}\mathbb L^{-n}T^n$.

Indeed, assume that $k_1=1$ then by applying the Nash inverse mapping theorem to $(x_1,\ldots,x_d)\mapsto(f(x_1,\ldots,x_d),x_2,\ldots,x_d)$ we get that $f$ is arc-analytically equivalent to $x_1$. In this case, we have that $\tilde Z_f(T)=0$ and we already know, by \cite[Proposition 8.3]{JBC2}, that if $\forall i,k_i\ge2$ then $\tilde Z_f(T)$ is non-zero.

We have just proved the following result. Assume that two Brieskorn polynomials $f$ and $g$ are arc-analytically equivalent. If $f$ is non-singular then $g$ is also non-singular.

From now on, we elude the non-singular case by assuming that $k_1\ge 2$.
\end{itemize}
\end{notation}

\begin{thm}[Main theorem]\label{thm:Main}
Consider two Brieskorn polynomials $$f(x)=\sum_{i=1}^d\varepsilon_ix_i^{k_i}\quad\quad\quad\text{ and }\quad\quad\quad g(x)=\sum_{i=1}^d\eta_ix_i^{l_i}$$ satisfying the following conditions:
\begin{enumerate}[nosep,label=(\alph*)]
\item $\varepsilon_i,\eta_i\in\{\pm1\}$
\item $2\le k_1\le\cdots\le k_d$ and $2\le l_1\le\cdots\le l_d$
\item $k_i=k_{i+1}=\cdots=k_{i+m}\Rightarrow\varepsilon_i\ge\cdots\ge\varepsilon_{i+m}$ and $l_i=l_{i+1}=\cdots=l_{i+m}\Rightarrow\eta_i\ge\cdots\ge\eta_{i+m}$
\end{enumerate}
Then the following are equivalent:
\begin{enumerate}[nosep,label=(\arabic*),ref=\ref{thm:Main}.(\arabic*)]
\item $f$ and $g$ are arc-analytically equivalent\label{item:fgAAE}
\item $Z_f(T)=Z_g(T)$\label{item:fgZETAsame}
\item\label{item:fgCONJ}
\begin{enumerate}[nosep,label=(\roman*),ref=\ref{item:fgCONJ}.(\roman*)]
\item $\forall i,\,k_i=l_i$ \label{item:fgSAMEexp}
\item For $j$ such that $k_j$ is even and not multiple of an odd exponent $k_m$, we have $\varepsilon_j=\eta_j$ \label{item:fgSAMEcoef}
\end{enumerate}
\end{enumerate}
\end{thm}
The previous statement was conjectured in \cite[Conjecture 1.10.1]{JBCThese}. It is compatible with the classifications of Brieskorn polynomials in two and three variables respectively by Koike--Parusiński \cite{KP03} and Fichou \cite{Fic05}.

\begin{rem}
We deduce from Theorem \ref{thm:Main} and the remarks in Notations \ref{not:BrieskSimpl} that the motivic zeta function defined in Definition \ref{defn:ZetaFun} is a complete arc-analytic invariant for the Brieskorn polynomials.
\end{rem}

\section{Proof of the main theorem}
\subsection{What is currently known}
The implication \ref{item:fgAAE}$\Rightarrow$\ref{item:fgZETAsame} derives from Theorem \ref{thm:Zaainv}. We already know that \ref{item:fgZETAsame}$\Rightarrow$\ref{item:fgSAMEexp} by \cite[Corollary 8.4]{JBC2}. The implication \ref{item:fgCONJ}$\Rightarrow$\ref{item:fgAAE} may be derived from the following proposition.
\begin{prop}[{\cite[Lemme 1.10.2]{JBCThese}}]\label{prop:KPgen}
Let $p\in\mathbb N_{>0}\setminus(2\mathbb N_{>0})$, $k\in\mathbb N_{>0}$, $m_1,\ldots,m_k\in\mathbb N_{>0}$ and $\varepsilon_1,\ldots,\varepsilon_k\in\{-1,1\}$. Then the polynomials $$f(x,y)=x^p+\sum_{i=1}^k\varepsilon_iy_i^{m_ip}\quad\quad\quad\text{ and }\quad\quad\quad g(x,y)=x^p+\sum_{i=1}^ky_i^{m_ip}$$ are arc-analytically equivalent.
\end{prop}
\begin{proof}
If $p=1$, we may conclude as for the last item of Notations \ref{not:BrieskSimpl}. Hence, we assume that $p\ge3$.

The following proof is inspired from \cite[p.2095]{KP03}.

Let $$h_s(x,y)=x^p+\sum_{i,\,\varepsilon_i=-1}\left(pxy_i^{m_i(p-1)}+sy_i^{m_ip}\right)+\sum_{i,\,\varepsilon_i=1}\left(pxy_i^{m_i(p-1)}+y_i^{m_ip}\right),\,s\in[-1,1]$$
Let $\alpha=\lcm(m_i)$ and define $\mu_i$ by $\alpha=m_i\mu_i$. For $s\in[-1,1]$, $h_s$ is weighted homogeneous with weights $(\alpha,\mu_1,\ldots,\mu_k)$ and has an isolated singularity at the origin. Indeed, $$\frac{1}{p}\partial_xh_s=x^{p-1}+\sum_{i=1}^ky_i^{m_i(p-1)}$$ is a sum of squares since $p-1$ is even.

Then, by Theorem \ref{thm:TFWeighted}, $h_{-1}$ and $h_1$ are arc-analytically equivalent.

Let $$l_s(x,y)=x^p+\sum_{i=1}^k\left(psxy_i^{m_i(p-1)}+\varepsilon_iy_i^{m_ip}\right),\,s\in[0,1]$$
Then, again by Theorem \ref{thm:TFWeighted}, $f=l_{0}$ and $h_{-1}=l_1$ are arc-analytically equivalent.

In the same way, we show that $h_{1}$ and $g$ are arc-analytically equivalent.

Hence we have $g\sim h_1$, $h_1\sim h_{-1}$ and $h_{-1}\sim f$, thus $f$ and $g$ are arc-analytically equivalent.
\end{proof}

Hence it remains to prove that \ref{item:fgZETAsame}$\Rightarrow$\ref{item:fgSAMEcoef}.

\subsection{Essence of the proof}
\begin{lemma}[{\cite[Example 6.10]{JBC2}}]\label{lem:OneMonomial}
Let $\varepsilon\in\{\pm1\}$ and $k\in\mathbb N_{>0}$. Then
\begin{align*}
\tilde Z_{\varepsilon x^{k}}(T)&=-T-\cdots-T^{k-1} \\
&\quad-\left(\mathbb 1-\left[\varepsilon x^k:\mathbb R^*\acts\mathbb R^*\rightarrow\mathbb R^*\right]\right)\mathbb L^{-1}T^{k}-\mathbb L^{-1}T^{k+1}-\ldots-\mathbb L^{-1}T^{2k-1} \\
&\quad-\left(\mathbb 1-\left[\varepsilon x^k:\mathbb R^*\acts\mathbb R^*\rightarrow\mathbb R^*\right]\right)\mathbb L^{-2}T^{2k}-\mathbb L^{-2}T^{2k+1}-\ldots-\mathbb L^{-2}T^{3k-1} \\
&\quad-\cdots \\
&=\sum_{m\ge1}\left[\varepsilon x^k:\mathbb R^*\acts\mathbb R^*\rightarrow\mathbb R^*\right]\mathbb L^{-m}T^{mk}-\sum_{n\ge1}\mathbb L^{-\left\lfloor\frac{n}{k}\right\rfloor}T^n
\end{align*}
where the action is given by $\lambda\cdot x=\lambda x$.
\end{lemma}

\begin{rem}
Notice that if $k$ is odd then $\mathbb 1-\left[\varepsilon x^k:\mathbb R^*\acts\mathbb R^*\rightarrow\mathbb R^*\right]=0$. Indeed, we have the following commutative diagram $$\xymatrix{\mathbb R^* \ar[rd]_{\varepsilon x^k} \ar[rr]^{\varepsilon x^k}_{\simeq,\AS} && \mathbb R^* \ar[ld]^{\id} \\ &\mathbb R^*&}$$ where the action on the right side is given by $\lambda\cdot x=\lambda^kx$.
\end{rem}

The rest of this section is devoted to the proof of \ref{item:fgZETAsame}$\Rightarrow$\ref{item:fgSAMEcoef}. For this purpose, we are going to prove that if $f$ is a singular Brieskorn polynomial, then for each monomial appearing in the expansion of $f$ with an even degree not multiple of an odd exponent, we are able to recover the sign of its coefficient from $Z_f(T)$. \\

We first fix some notation. Let $$f(x)=\sum_{i=1}^d\varepsilon_ix_i^{k_i}$$ be a Brieskorn polynomial as in the statement of Theorem \ref{thm:Main}. \\

We denote by $a_n$ the coefficients of $\tilde Z_f(T)$ so that $$\tilde Z_f(T)=\sum_{n\ge1}a_nT^n$$

Notice that by \cite[Proposition 8.3]{JBC2}, we already know how to express $(k_1,\ldots,k_d)$ in terms of the coefficients of $\tilde Z_f(T)$. \\

We denote by $K$ the set of even exponents which are not multiple of an odd exponent, i.e. $$K=\left\{k_i,\,\forall j\in\llbracket1,d\rrbracket,\,k_j\in(2\mathbb N+1)\Rightarrow k_j\nmid k_i\right\}$$ We assume that $K=\left\{k_{i_1},\ldots,k_{i_s}\right\}$ with $$k_{i_1}<k_{i_2}<\cdots<k_{i_s}$$ For $k\in K$ and $\varepsilon\in\{+,-\}$, we set $\sigma_k^\varepsilon=\#\left\{i,\,k_i=k,\,\varepsilon_i=\varepsilon1\right\}$, i.e. $\sigma_k^+$ (resp. $\sigma_k^-$) is the number of positive (resp. negative) coefficients of degree $k$. Our goal is to deduce these $\sigma_{k}^\varepsilon$ from $Z_f(T)$, or equivalently from $\tilde Z_f(T)$.

\begin{lemma}\label{lem:recovBeta}
Let $k\in K$. Then
$$\beta\left(x\in\mathbb R^{\{i,\,k_i|k\}},\,\sum_{i,\,k_i|k}\varepsilon_ix_i^{k_i}=0\right)=u^{\#\{i,\,k_i|k\}-1}-\beta\left(\overline{a_k}\right)u^{\sum_{i=1}^d\left\lfloor\frac{k}{k_i}\right\rfloor-1}$$
and, for $\varepsilon=+,-$,
$$\beta\left(x\in\mathbb R^{\{i,\,k_i|k\}},\,\sum_{i,\,k_i|k}\varepsilon_ix_i^{k_i}=\varepsilon1\right)=\left(u\beta\left(F^\varepsilon(a_k)\right)-\beta\left(\overline{a_k}\right)\right)u^{\sum_{i=1}^d\left\lfloor\frac{k}{k_i}\right\rfloor-1}+u^{\#\{i,\,k_i|k\}-1}$$
where, for $I\subset\{1,\ldots,d\}$, we set $\mathbb R^I=\left\{(x_1,\ldots,x_d)\in\mathbb R^d,\,\forall j\notin I,\,x_j=0\right\}$.
\end{lemma}
\begin{proof}
From Theorem \ref{thm:convolutionf} and Lemma \ref{lem:OneMonomial}, we get

\begin{align*}
a_k\mathbb L^{\sum_{i=1}^d\left\lfloor\frac{k}{k_i}\right\rfloor}&=-\bigast_{i,\,k_i|k}\left(\mathbb 1-\left[\varepsilon_ix_i^{k_i}:\mathbb R^*\acts\mathbb R^*\rightarrow\mathbb R^*\right]\right) \\
&=\sum_{I\subset\{i,\,k_i|k\}}(-1)^{|I|}\bigast_{i\in I}\left[\varepsilon_ix_i^{k_i}:\mathbb R^*\acts\mathbb R^*\rightarrow\mathbb R^*\right] \\
&=\sum_{I\subset\{i,\,k_i|k\}}\left(\left[x\in(\mathbb R^*)^I,\,\sum_{i\in I}\varepsilon_ix_i^{k_i}\neq0\right]-\left[(x,y)\in(\mathbb R^*)^I\times\mathbb R^*,\,\sum_{i\in I}\varepsilon_ix_i^{k_i}=0\right]\right)
\end{align*}
where $I$ is possibly empty. Notice it is also possible to deduce this formula from \cite[Proposition 4.8]{JBC4}. \\

Hence, by additivity of the virtual Poincaré polynomial,
\begin{align*}
\hspace{-1cm}\beta\left(\overline{a_k}\right)u^{\sum_{i=1}^d\left\lfloor\frac{k}{k_i}\right\rfloor}&=\sum_{I\subset\{i,\,k_i|k\}}\left(\beta\left(x\in(\mathbb R^*)^I,\,\sum_{i\in I}\varepsilon_ix_i^{k_i}\neq0\right)-(u-1)\beta\left(x\in(\mathbb R^*)^I,\,\sum_{i\in I}\varepsilon_ix_i^{k_i}=0\right)\right) \\
&=\beta\left(x\in\mathbb R^{\{i,\,k_i|k\}},\,\sum_{i,\,k_i|k}\varepsilon_ix_i^{k_i}\neq0\right)-(u-1)\beta\left(x\in\mathbb R^{\{i,\,k_i|k\}},\,\sum_{i,\,k_i|k}\varepsilon_ix_i^{k_i}=0\right)\\
&=u^{\#\{i,\,k_i|k\}}-u\beta\left(x\in\mathbb R^{\{i,\,k_i|k\}},\,\sum_{i,\,k_i|k}\varepsilon_ix_i^{k_i}=0\right)
\end{align*}
Thus $$\beta\left(x\in\mathbb R^{\{i,\,k_i|k\}},\,\sum_{i,\,k_i|k}\varepsilon_ix_i^{k_i}=0\right)=u^{\#\{i,\,k_i|k\}-1}-\beta\left(\overline{a_k}\right)u^{\sum_{i=1}^d\left\lfloor\frac{k}{k_i}\right\rfloor-1}$$
Similarly, we get
$$\beta\left(F^\varepsilon(a_k)\right)u^{\sum_{i=1}^d\left\lfloor\frac{k}{k_i}\right\rfloor}=\beta\left(x\in\mathbb R^{\{i,\,k_i|k\}},\,\sum_{i,\,k_i|k}\varepsilon_ix_i^{k_i}=\varepsilon1\right)-\beta\left(x\in\mathbb R^{\{i,\,k_i|k\}},\,\sum_{i,\,k_i|k}\varepsilon_ix_i^{k_i}=0\right)$$
So that,
$$\beta\left(x\in\mathbb R^{\{i,\,k_i|k\}},\,\sum_{i,\,k_i|k}\varepsilon_ix_i^{k_i}=\varepsilon1\right)=\left(u\beta\left(F^\varepsilon(a_k)\right)-\beta\left(\overline{a_k}\right)\right)u^{\sum_{i=1}^d\left\lfloor\frac{k}{k_i}\right\rfloor-1}+u^{\#\{i,\,k_i|k\}-1}$$
\end{proof}

\begin{proof}[Proof of \ref{item:fgZETAsame}$\Rightarrow$\ref{item:fgSAMEcoef}]
\ \\
We are going to compute inductively $\sigma_{k_{i_r}}^\pm$ for $r=1,\ldots,s$ in terms of the coefficients of $\tilde Z_f(T)$, which is enough to conclude.

Assume that $\sigma_{k_{i_1}}^\pm,\ldots,\sigma_{k_{i_{r-1}}}^\pm$ are already known. We are going to compute $\sigma_{k_{i_r}}^\pm$. Notice that the following argument allows one to compute directly $\sigma_{k_{i_1}}^\pm$.

By Lemma \ref{lem:recovBeta}, we may express $\displaystyle\pi=\beta\left(x\in\mathbb R^{\{i,\,k_i|k_{i_r}\}},\,\sum_{i,\,k_i|k_{i_r}}\varepsilon_ix_i^{k_i}=1\right)$ in terms of the coefficients of $Z_f(T)$.

For an exponent $k_i$ dividing $k_{i_r}$, we write $k_i=2^{N_i}l_i$ with $l_i$ odd. Notice that by definition of $K$, $k_i\in K$ and $N_i>0$. Then, using the $\AS$-change of variables $\tilde{x_i}=x_i^{l_i}$, we get
$$\pi=\beta\left(x\in\mathbb R^{\{i,\,k_i|k_{i_r}\}},\,\sum_{i,\,k_i|k_{i_r}}\varepsilon_ix_i^{k_i}=1\right)=\beta\left(x\in\mathbb R^{\{i,\,k_i|k_{i_r}\}},\,\sum_{i,\,k_i|k_{i_r}}\varepsilon_i\tilde{x_i}^{2^{N_i}}=1\right)$$

By Theorem \ref{thm:vppfibers}, if the leading coefficient of $\rho:=\pi-u^{\#\{i,\,k_i|k_{i_r}\}-1}$ is positive then
$$\deg(\rho)=\sum_{k\in K,\,k|k_{i_r}}\sigma_{k}^-\quad\text{and thus}\quad\sigma_{k_{i_r}}^-=\deg(\rho)-\sum_{\substack{k\in K,\,k|k_{i_r}\\k\neq k_{i_r}}}\sigma_{k}^-$$

Otherwise, still by Theorem \ref{thm:vppfibers}, if the leading coefficient of $\rho$ is negative, then $$\deg(\rho)=\sum_{k\in K,\,k|k_{i_r}}\sigma_{k}^--1\quad\text{and thus}\quad\sigma_{k_{i_r}}^-=\deg(\rho)+1-\sum_{\substack{k\in K,\,k|k_{i_r}\\k\neq k_{i_r}}}\sigma_{k}^-$$

Since we already know $\sigma_{k_{i_j}}^-$ for $j<r$, we are able to compute $\sigma^-_{k_{i_r}}$. \\

Finally, $\sigma^+_{k_{i_r}}=\#\left\{i\in\llbracket1,d\rrbracket,\,k_i=k_{i_r}\right\}-\sigma^-_{k_{i_r}}$.
\end{proof}

\subsection{Some virtual Poincaré polynomials}\label{sub:SomeLemmas}
\begin{lemma}\label{lem:BETAdiff}
Let $k\in2\mathbb N_{>0}$ and $g\in\mathbb R[y_1,\ldots,y_d]$ ($d$ possibly equals 0). Then $$\beta\left((x_1,x_2,y_1,\ldots,y_d)\in\mathbb R^{2+d},\,x_1^k-x_2^k+g(y)=0\right)=u\beta\left((y_1,\ldots,y_d)\in\mathbb R^{d},\,g(y)=0\right)+(u-1)u^d$$
\end{lemma}
\begin{proof}
Assume that $k=2^Nl$ where $l$ is odd, then the map $\mathbb R^{2+d}\rightarrow\mathbb R^{2+d}$ defined by $(x_1,x_2,y_1,\ldots,y_d)\mapsto(x_1^l,x_2^l,y_1,\ldots,y_d)$ is an $\AS$-bijection. Hence $$\beta\left((x_1,x_2,y_1,\ldots,y_d)\in\mathbb R^{2+d},\,x_1^k-x_2^k+g(y)=0\right)=\beta\left((z_1,z_2,y_1,\ldots,y_d)\in\mathbb R^{2+d},\,z_1^{2^N}-z_2^{2^N}+g(y)=0\right)$$
Now, notice that $$z_1^{2^N}-z_2^{2^N}=(z_1-z_2)\prod_{j=0}^{N-1}\left(z_1^{2^j}+z_2^{2^j}\right)$$
Since $N\ge1$, $(z_1,z_2)\mapsto\left(z_1-z_2,\prod_{j=0}^{N-1}\left(z_1^{2^j}+z_2^{2^j}\right)\right)$ is a bijection with $\AS$-graph as explained in what follows. First its graph is in $\AS$ since it is a polynomial mapping. For the bijectivity, fix $t=z_1-z_2$. If $t\neq0$, then $$\prod_{j=0}^{N-1}\left(z_1^{2^j}+z_2^{2^j}\right)=\frac{(z_2+t)^{2^N}-z_2^{2^N}}{t}$$ and the function $\mathbb R\rightarrow\mathbb R$ defined by $z_2\rightarrow\frac{(z_2+t)^{2^N}-z_2^{2^N}}{t}$ is a bijection. Indeed, it is surjective as a polynomial of odd degree and it is one-to-one since it is strictly increasing as one can be convinced by noticing that its derivative $\frac{2^N}{t}\left((z_2+t)^{2^N-1}-z_2^{2^N-1}\right)$ is positive (as a polynomial of even degree with positive leading coefficient and no real root).  When $t=0$, we get that $$\prod_{j=0}^{N-1}\left(z_1^{2^j}+z_2^{2^j}\right)=2^Nz_2^{2^N-1}$$ which also defines a bijection with $\AS$-graph.

Hence
$$\beta\left((z_1,z_2,y_1,\ldots,y_d)\in\mathbb R^{2+d},\,z_1^{2^N}-z_2^{2^N}+g(y)=0\right)=\beta\left((u_1,u_2,y_1,\ldots,y_d)\in\mathbb R^{2+d},\,u_1u_2+g(y)=0\right)$$
Since $(u_1,y_1,\ldots,y_d)\mapsto\left(u_1,u_2=-\frac{g(y)}{u_1},y\right)$ is a bijection with $\AS$-graph from $\mathbb R^*\times\mathbb R^d$ to $\left\{(u_1,u_2,y_1,\ldots,y_d)\in\mathbb R^*\times\mathbb R^{1+d},\,u_1u_2+g(y)=0\right\}$, we have
\begin{align*}
\hspace{-.7cm}\beta\left((u_1,u_2,y_1,\ldots,y_d)\in\mathbb R^{2+d},\,u_1u_2+g(y)=0\right)&=\beta\left((0,u_2,y_1,\ldots,y_d)\in\mathbb R^{2+d},\,g(y)=0\right)+\beta(\mathbb R^*\times\mathbb R^d) \\
&=u\beta\left((y_1,\ldots,y_d)\in\mathbb R^{d},\,g(y)=0\right)+(u-1)u^d
\end{align*}
\end{proof}

\begin{lemma}\label{lem:BETAsame}
Let $$f=\varepsilon\sum_{i=1}^{r}x_{i}^{2^N}+\sum_{j=1}^s\varepsilon_jy_j^{2^{k_j}}$$
with $r,s\ge0$, $\varepsilon,\varepsilon_j\in\{\pm1\}$ and $k_j<N$. Then
$$\beta\left((x,y)\in\mathbb R^r\times\mathbb R^s,\,f(x,y)=0\right)=(u^r-1)\beta\left(y\in\mathbb R^s,\,\sum_{j=1}^s\varepsilon_jy_j^{2^{k_j}}=-\varepsilon\right)+\beta\left(y\in\mathbb R^s,\,\sum_{j=1}^s\varepsilon_jy_j^{2^{k_j}}=0\right)
$$
\end{lemma}
\begin{proof}
The map $(\mathbb R^r\setminus\{0\})\times\mathbb R^s\mapsto(\mathbb R^r\setminus\{0\})\times\mathbb R^s$ defined by $(x,y)\mapsto(x,z)=\left(x,\left(\sum_{i=1}^{r}x_{i}^{2^N}\right)^{-2^{-k_j}}y_i\right)$ is a bijection with $\AS$-graph. Indeed, since $k_j<N$, the function $x\mapsto\left(\sum_{i=1}^{r}x_{i}^{2^N}\right)^{2^{-k_j}}$ is arc-analytic at the origin and at the infinity.

Hence
\begin{align*}
&\beta\left((x,y)\in\mathbb R^r\times\mathbb R^s,\,f(x,y)=0\right)\\
&\quad\quad=\beta\left((x,y)\in(\mathbb R^r\setminus\{0\})\times\mathbb R^s,\,f(x,y)=0\right)+\beta\left(y\in\mathbb R^s,\,\sum_{j=1}^s\varepsilon_jy_j^{2^{k_j}}=0\right) \\
&\quad\quad=\beta\left((x,z)\in(\mathbb R^r\setminus\{0\})\times\mathbb R^s,\,\varepsilon+\sum_{j=1}^s\varepsilon_jz_j^{2^{k_j}}=0\right)+\beta\left(y\in\mathbb R^s,\,\sum_{j=1}^s\varepsilon_jy_j^{2^{k_j}}=0\right) \\
&\quad\quad=(u^r-1)\beta\left(z\in\mathbb R^s,\,\varepsilon+\sum_{j=1}^s\varepsilon_jz_j^{2^{k_j}}=0\right)+\beta\left(y\in\mathbb R^s,\,\sum_{j=1}^s\varepsilon_jy_j^{2^{k_j}}=0\right)
\end{align*}
\end{proof}

\noindent Until Lemma \ref{lem:BETAfib-1}, we are going to use the following notations: $$P_{A,B}(x,y)=\sum_{i=1}^Ax_i^{2^N}-\sum_{j=1}^By_j^{2^N}$$ and $$R(z)=\sum_{i=1}^s\varepsilon_iz_j^{2^{k_i}}$$ with
\begin{itemize}[nosep]
\item $(A,B)\in\mathbb N_{\ge0}^2\setminus\{(0,0)\}$
\item $\varepsilon_i\in\{\pm1\}$
\item $s\ge0$
\item $1\le k_i<N$
\end{itemize}
\ \\
In order to lighten the notations, if $f\in\mathbb R[x_1,\ldots,x_d]$, we will simply write $\beta(f=\varepsilon)$ for $\beta\left(x\in\mathbb R^d,\,f(x)=\varepsilon\right)$.

\begin{lemma}\label{lem:BETAone}
\begin{itemize}[nosep]
\item $\displaystyle\beta(P_{A+1,B}+R=0)=(u-1)\beta(P_{A,B}+R=-1)+\beta(P_{A,B}+R=0)$
\item $\displaystyle\beta(P_{A,B+1}+R=0)=(u-1)\beta(P_{A,B}+R=1)+\beta(P_{A,B}+R=0)$
\end{itemize}
\end{lemma}
\begin{proof}
We first rewrite:
\begin{align*}
\beta(P_{A+1,B}+R=0)&=\beta\left((w,x,y,z)\in\mathbb R\times\mathbb R^A\times\mathbb R^B\times\mathbb R^s,\,w^{2^N}+P_{A,B}(x,y)+R(z)=0\right) \\
&=\beta\left((w,x,y,z)\in\mathbb R^*\times\mathbb R^A\times\mathbb R^B\times\mathbb R^s,\,w^{2^N}+P_{A,B}(x,y)+R(z)=0\right)\\
&\quad\quad+\beta\left(P_{A,B}+R=0\right)
\end{align*}
To compute the first term of the RHS, we use the following change of variables: $$w=\tilde w,x_i=w\tilde x_i,y_j=w\tilde y_j,z_i=w^{2^{N-k_i}}\tilde{z_i}$$
Hence
\begin{align*}
\beta(P_{A+1,B}+R=0)&=\beta\left((\tilde w,\tilde x,\tilde y,\tilde z)\in\mathbb R^*\times\mathbb R^A\times\mathbb R^B\times\mathbb R^s,\,1+P_{A,B}(\tilde x,\tilde y)+R(\tilde z)=0\right)\\
&\quad\quad+\beta\left(P_{A,B}+R=0\right) \\
&=(u-1)\beta(P_{A,B}+R=-1)+\beta(P_{A,B}+R=0)
\end{align*}
We obtain the second equality, noticing that $$\beta(P_{A,B+1}+R=0)=\beta(P_{B+1,A}+(-R)=0)$$
\end{proof}

\begin{lemma}\label{lem:BETAfib0}
$$\beta(P_{A,B}+R=0)-u^{A+B+s-1}=\left\{\begin{array}{ll}(u^B-u^A)\beta(R=1)+u^A\beta(R=0)-u^{B+s-1}&\text{ if $B\ge A$}\\(u^A-u^B)\beta(R=-1)+u^B\beta(R=0)-u^{A+s-1}&\text{ if $A\ge B$}\end{array}\right.$$
\end{lemma}
\begin{proof}
Assume that $B\ge A$. \\
Set $\alpha_{A,B}=\beta(P_{A,B}+R=0)-u^{A+B+s-1}$. \\
By Lemma \ref{lem:BETAdiff}, $\alpha_{A,B}=u\alpha_{A-1,B-1}$. Hence,
\begin{align*}
\alpha_{A,B}&=u^A\alpha_{0,B-A} \\
&=u^A\left(\beta(P_{0,B-A}+R=0)-u^{B-A+s-1})\right)
\end{align*}
By Lemma \ref{lem:BETAsame},
\begin{align*}
\alpha_{A,B}&=u^A\left((u^{B-A}-1)\beta(R=1)+\beta(R=0)-u^{B-A+s-1}\right) \\
&=(u^B-u^A)\beta(R=1)+u^A\beta(R=0)-u^{B+s-1}
\end{align*}
If $A\ge B$, we reduce to the previous case noticing that $$\beta(P_{A,B}+R=0)=\beta(P_{B,A}+(-R)=0)$$
\end{proof}

\begin{lemma}\label{lem:BETAfib1}
$$\beta(P_{A,B}+R=1)-u^{A+B+s-1}=\left\{\begin{array}{ll}u^{B-1}\left(u\beta(R=1)-u^s\right)&\text{ if $B\ge A$}\\u^B\left(\beta(R=0)-\beta(R=-1)\right)&\text{ if $A>B$}\end{array}\right.$$
\end{lemma}
\begin{proof}
First, notice that by Lemma \ref{lem:BETAone}, $$(u-1)\beta(P_{A,B}+R=1)=\beta(P_{A,B+1}+R=0)-\beta(P_{A,B}+R=0)$$

Next, by Lemma \ref{lem:BETAfib0},
$$(u-1)\beta(P_{A,B}+R=1)=
\left\{\begin{array}{ll}
u^B(u-1)\beta(R=1)-u^{B+s-1}(u-1)+u^{A+B+s-1}(u-1)          &\text{ if $B\ge A$} \\
-u^B(u-1)\beta(R=-1)+u^B(u-1)\beta(R=0)+u^{A+B+s-1}(u-1) &\text{ if $A>B$}
\end{array}\right.$$
\end{proof}

\begin{lemma}\label{lem:BETAfib-1}
$$\beta(P_{A,B}+R=-1)-u^{A+B+s-1}=\left\{\begin{array}{ll}u^{A-1}\left(u\beta(R=-1)-u^s\right)&\text{ if $A\ge B$}\\u^A\left(\beta(R=0)-\beta(R=1)\right)&\text{ if $B>A$}\end{array}\right.$$
\end{lemma}
\begin{proof}
Notice that $\beta(P_{A,B}+R=-1)=\beta(P_{B,A}+(-R)=1)$.
\end{proof}

The following formulae generalize the ones obtained by G. Fichou for the homogeneous case of degree $2$ \cite[Proposition 2.1, Corollaries 2.5\&2.6]{Fic06} \cite[Proposition 2.1]{Fic06-rims}.
\begin{thm}\label{thm:vppfibers}
Fix $1\le k_1<\cdots<k_d$ with $d\in\mathbb N_{\ge1}$.

\noindent For $r\in\{1,\ldots,d\}$, fix $(A_r,B_r)\in\mathbb N_{\ge0}^2\setminus\{(0,0)\}$ and define $P_r\in\mathbb R[x_{r,1},\ldots,x_{r,A_r},y_{r,1},\ldots,y_{r,B_r}]$ by $$P_r(x_r,y_r)=\sum_{i=1}^{A_r}x_{r,i}^{2^{k_r}}-\sum_{j=1}^{B_r}y_{r,j}^{2^{k_r}}$$

\noindent Define $f\in\mathbb R[x_{r,i},y_{r,j}]_{1\le r\le d}$ by $f=\sum_{r=1}^dP_r$.

\noindent Set $\sigma^+=\sum_{r=1}^dA_r$, $\sigma^-=\sum_{r=1}^dB_r$ and $s=\sigma^++\sigma^-$ the number of variables.

\noindent Set $m=\min\left(\{r,\,A_r\neq B_r\}\right)$ with the convention $\min(\emptyset)=\infty$.

\noindent Then
$$\beta\left(f=0\right)-u^{s-1}=\left\{\begin{array}{ll}-u^{\sigma^+-1}+u^{\sigma^-}&\text{ if $m\neq\infty$ and $A_m>B_m$}\\u^{\sigma^+}-u^{\sigma^--1}&\text{ otherwise}\end{array}\right.$$
$$\beta\left(f=1\right)-u^{s-1}=\left\{\begin{array}{ll}u^{\sigma^-}&\text{ if $m\neq\infty$ and $A_m>B_m$}\\-u^{\sigma^--1}&\text{ otherwise}\end{array}\right.$$
$$\beta\left(f=-1\right)-u^{s-1}=\left\{\begin{array}{ll}u^{\sigma^+}&\text{ if $m\neq\infty$ and $B_m>A_m$}\\-u^{\sigma^+-1}&\text{ otherwise}\end{array}\right.$$
\end{thm}
\begin{proof}
We are going to prove the theorem by induction on $d$ (i.e. on the number of different exponents appearing in the expansion of $f$). \\

When $d=1$, it is a direct consequence of Lemmas \ref{lem:BETAfib0}, \ref{lem:BETAfib1} and \ref{lem:BETAfib-1} (case $s=0$). \\

Now, assume that the theorem is proved for some $d\ge1$ and define $f$ as in the statement. Let $k_{d+1}>k_d$ and $(A_{d+1},B_{d+1})\in\mathbb N_{\ge0}^2\setminus\{(0,0)\}$. To lighten the notations, we set $(A,B):=(A_{d+1},B_{d+1})$. We are going to compute $\beta\left(P_{d+1}+f=\varepsilon\right)-u^{A+B+s-1},\,\varepsilon\in\{-1,0,1\}$. \\

\begin{enumerate}[nosep,label=\arabic*.]
\item \underline{First case:} $m=\infty$. Notice that in this case, we have $\sigma^+=\sigma^-$.
\begin{enumerate}[nosep,label=(\alph*)]
\item Assume that $A=B$. By Lemma \ref{lem:BETAfib0}, we have
\begin{align*}
\beta\left(P_{d+1}+f=0\right)-u^{A+B+s-1}&= u^A\left(\beta(f=0)-u^{s-1}\right)\\
&=u^{A+\sigma^+}-u^{B+\sigma^--1}
\end{align*}
By Lemma \ref{lem:BETAfib1}, we have
\begin{align*}
\beta\left(P_{d+1}+f=1\right)-u^{A+B+s-1}&= u^B\left(\beta(f=1)-u^{s-1}\right) \\ &=-u^{B+\sigma^--1}
\end{align*}By Lemma \ref{lem:BETAfib-1}, we have
\begin{align*}
\beta\left(P_{d+1}+f=-1\right)-u^{A+B+s-1} &= u^A\left(\beta(f=-1)-u^{s-1}\right)\\ &=-u^{A+\sigma^+-1}
\end{align*}
\item  Assume that $A>B$. By Lemma \ref{lem:BETAfib0}, we have
\begin{align*}
\beta\left(P_{d+1}+f=0\right)-u^{A+B+s-1}&=u^A\left(\beta(f=-1)-u^{s-1}\right)+u^B\left(\beta(f=0)-\beta(f=-1)\right) \\
&=-u^{A+\sigma^+-1}+u^{B+\sigma^-}
\end{align*}
By Lemma \ref{lem:BETAfib1}, we have
\begin{align*}
\beta\left(P_{d+1}+f=1\right)-u^{A+B+s-1}&=u^B\left(\beta(f=0)-\beta(f=-1)\right) \\ &= u^{B+\sigma^-}
\end{align*}
By Lemma \ref{lem:BETAfib-1}, we have
\begin{align*}
\beta\left(P_{d+1}+f=-1\right)-u^{A+B+s-1}&=u^A\left(\beta(f=-1)-u^{s-1}\right) \\ &= -u^{A+\sigma^+-1}
\end{align*}
\item  Assume that $A<B$. By Lemma \ref{lem:BETAfib0}, we have
\begin{align*}
\beta\left(P_{d+1}+f=0\right)-u^{A+B+s-1}&=u^A\left(\beta(f=0)-\beta(f=1)\right)+u^B\left(\beta(f=1)-u^{s-1}\right) \\
&=u^{A+\sigma^+}-u^{B+\sigma^--1}
\end{align*}
By Lemma \ref{lem:BETAfib1}, we have
\begin{align*}
\beta\left(P_{d+1}+f=1\right)-u^{A+B+s-1}&= u^B\left(\beta(f=1)-u^{s-1}\right)\\ &= -u^{B+\sigma^--1}
\end{align*}
By Lemma \ref{lem:BETAfib-1}, we have
\begin{align*}
\beta\left(P_{d+1}+f=-1\right)-u^{A+B+s-1}&=u^A\left(\beta(f=0)-\beta(f=1)\right) \\ &= u^{A+\sigma^+}
\end{align*}
\end{enumerate}

\item \underline{Second case:} $1\le m\le d$. Assume that $A_m>B_m$ (the other case is similar).
\begin{enumerate}[nosep,label=(\alph*)]
\item Assume that $A=B$: \\
By Lemma \ref{lem:BETAfib0}, we have
\begin{align*}
\beta\left(P_{d+1}+f=0\right)-u^{A+B+s-1}&= u^A\left(\beta(f=0)-u^{s-1}\right)\\
&= -u^{A+\sigma^+-1}+u^{B+\sigma^-}
\end{align*}
By Lemma \ref{lem:BETAfib1}, we have
\begin{align*}
\beta\left(P_{d+1}+f=1\right)-u^{A+B+s-1}&= u^B\left(\beta(f=1)-u^{s-1}\right) \\ &=u^{B+\sigma^-}
\end{align*}By Lemma \ref{lem:BETAfib-1}, we have
\begin{align*}
\beta\left(P_{d+1}+f=-1\right)-u^{A+B+s-1} &= u^A\left(\beta(f=-1)-u^{s-1}\right)\\ &=-u^{A+\sigma^+-1}
\end{align*}
\item Assume that $A>B$. \\
By Lemma \ref{lem:BETAfib0}, we have
\begin{align*}
\beta\left(P_{d+1}+f=0\right)-u^{A+B+s-1}&=u^A\left(\beta(f=-1)-u^{s-1}\right)+u^B\left(\beta(f=0)-\beta(f=-1)\right) \\
&=-u^{A+\sigma^+-1}+u^{B+\sigma^-}
\end{align*}
By Lemma \ref{lem:BETAfib1}, we have
\begin{align*}
\beta\left(P_{d+1}+f=1\right)-u^{A+B+s-1}&=u^B\left(\beta(f=0)-\beta(f=-1)\right) \\ &= u^{B+\sigma^-}
\end{align*}
By Lemma \ref{lem:BETAfib-1}, we have
\begin{align*}
\beta\left(P_{d+1}+f=-1\right)-u^{A+B+s-1}&=u^A\left(\beta(f=-1)-u^{s-1}\right) \\ &=-u^{A+\sigma^+-1}
\end{align*}
\item Assume that $A<B$. \\
By Lemma \ref{lem:BETAfib0}, we have
\begin{align*}
\beta\left(P_{d+1}+f=0\right)-u^{A+B+s-1}&=u^A\left(\beta(f=0)-\beta(f=1)\right)+u^B\left(\beta(f=1)-u^{s-1}\right) \\
&=-u^{A+\sigma^+-}+u^{B+\sigma^-}
\end{align*}
By Lemma \ref{lem:BETAfib1}, we have
\begin{align*}
\beta\left(P_{d+1}+f=1\right)-u^{A+B+s-1}&=u^B\left(\beta(f=1)-u^{s-1}\right) \\ &= u^{B+\sigma^-}
\end{align*}
By Lemma \ref{lem:BETAfib-1}, we have
\begin{align*}
\beta\left(P_{d+1}+f=-1\right)-u^{A+B+s-1}&=u^A\left(\beta(f=0)-\beta(f=1)\right) \\ &=-u^{A+\sigma^+-1}
\end{align*}
\end{enumerate}
\end{enumerate}
\end{proof}

By evaluating in $u=-1$ the previous formulae, we obtain the following corollary.
\begin{cor}
$$\chi_c(f=0)=(-1)^{s-1}+(-1)^{\sigma^+}+(-1)^{\sigma^-}$$
$$\chi_c(f=1)=(-1)^{s-1}+(-1)^{\sigma^-}$$
$$\chi_c(f=-1)=(-1)^{s-1}+(-1)^{\sigma^+}$$
\end{cor}

\end{document}